\documentclass{amsart}
\usepackage{graphicx}
\usepackage{amssymb}
\usepackage[shortlabels]{enumitem}

\newtheorem{thm}{Theorem}[section]

\newtheorem{prop}[thm]{Proposition}
\theoremstyle{definition}

\theoremstyle{remark}
\newtheorem{rem}[thm]{Remark}

\begin{document}

\title{The anisotropic Bernstein problem}
\author{Connor Mooney}
\address{Department of Mathematics, UC Irvine}
\email{\tt mooneycr@math.uci.edu}
\author{Yang Yang}
\address{Department of Mathematics, UC Irvine}
\email{\tt y.yang@uci.edu}

\begin{abstract}
We construct nonlinear entire anisotropic minimal graphs over $\mathbb{R}^4$, completing the solution to the anisotropic Bernstein problem. The examples we construct have a variety of growth rates, and our approach both generalizes to higher dimensions and recovers and elucidates known examples of nonlinear entire minimal graphs over $\mathbb{R}^{n},\, n \geq 8$.
\end{abstract}

\maketitle

\section{Introduction}
In this paper we study graphical critical points of parametric elliptic functionals, which assign to an oriented hypersurface $\Sigma \subset \mathbb{R}^{n+1}$ the value
\begin{equation}\label{PEF}
A_{\Phi}(\Sigma) = \int_{\Sigma} \Phi(\nu)\,dA.
\end{equation}
Here $\nu$ is the unit normal to $\Sigma$, and $\Phi$ is a uniformly elliptic integrand, namely, a one-homogeneous function on $\mathbb{R}^{n+1}$ that is positive and smooth on $\mathbb{S}^n$, and satisfies in addition that $\{\Phi < 1\}$ is uniformly convex. Such functionals have attracted recent attention both for their applied and theoretical interest (\cite{CL1}, \cite{DD}, \cite{DDG}, \cite{DM}, \cite{DT}, \cite{DMMN}, \cite{FM}, \cite{FMP}). In particular, they arise in models of crystal surfaces and in Finsler geometry, and they present important technical challenges not present for the area functional (especially due to the lack of a monotonicity formula, see \cite{All}), often leading to more general and illuminating proofs even in the area case.

The anisotropic Bernstein problem asks whether critical points of $A_{\Phi}$ which are graphs of functions defined on all of $\mathbb{R}^n$ are necessarily hyperplanes. In the case of the area functional $\Phi(x) = |x|$, it is known through spectacular work of Bernstein, Fleming \cite{F}, De Giorgi \cite{DeG}, Almgren \cite{A}, Simons \cite{S}, and Bombieri-De Giorgi-Giusti \cite{BDG} that the answer is positive if and only if $n \leq 7$. For general uniformly elliptic integrands, it is known that the answer is positive in dimension $n = 2$ by work of Jenkins \cite{J} and in dimension $n = 3$ by work of Simon \cite{Si2}. It is also known by recent work of the first author \cite{M} that the answer is negative in dimensions $n \geq 6$. This left open the cases $n = 4,\,5$. The purpose of this paper is to settle the anisotropic Bernstein problem negatively in these remaining cases:

\begin{thm}\label{Main}
There exists a smooth nonlinear function $u: \mathbb{R}^4 \rightarrow \mathbb{R}$ and a uniformly elliptic integrand $\Phi$ on $\mathbb{R}^5$ such that the graph of $u$ in $\mathbb{R}^5$ minimizes $A_{\Phi}$.
\end{thm}

\noindent We in fact construct, for any $\mu \in (0,\,1/2)$, a pair $(u,\,\Phi)$ proving Theorem \ref{Main} such that 
$$\sup_{B_r} u \sim r^{1 + \mu}$$ 
for $r$ large, and our methods both generalize to higher dimensions and shed light on known examples in the case of the area functional.


Here and below we denote for $x \in \mathbb{R}^{k+1}$ and $y \in \mathbb{R}^{l+1}$ the cone $C_{kl} \subset \mathbb{R}^{k+l+2}$ by 
$$C_{kl} = \{|x| = |y|\}.$$
The first examples of nonlinear entire minimal graphs were constructed in \cite{BDG}. The examples in \cite{BDG} are asymptotic to $C_{kk} \times \mathbb{R} \subset \mathbb{R}^{2k+3}$, for $k \geq 3$. Similarly, for the more general anisotropic case, the example in \cite{M} is asymptotic to $C_{22} \times \mathbb{R}$. However, the approaches in \cite{BDG} and \cite{M} are completely different. In the former, the method is to carefully construct super- and sub-solutions to the minimal surface equation, with appropriate ordering and symmetries, and then use the maximum principle. In \cite{M}, the method is to first fix a choice of solution $u$, and then construct the integrand $\Phi$ by solving a linear {\it hyperbolic} equation. In the case $k = l = 2$ it turned out that for the simple choice
$$u = |x|^2 - |y|^2,$$ 
the hyperbolic equation for $\Phi$ reduced after a change of variable to the classical two-variable wave equation. This made the construction significantly shorter and more elementary than for the case of the area functional. Unfortunately the analogous choice for $u$ in the case $k = l = 1$ does {\it not} solve an equation of minimal surface type, as shown in \cite{M}. However, the choice
$$u = |x|^{4/3} - |y|^{4/3}$$
also yields a two-variable wave equation for $\Phi$ in the case $k = l = 1$ after a change of variable, so it seemed likely that the methods in \cite{M} could be adapted. The issue is that this choice of $u$ does not solve an equation of minimal surface type near $\{|x||y| = 0\}$, so both $u$ and the integrand $\Phi$ obtained by solving the wave equation need to be modified. This seems to be tricky, and we are leaving the pursuit of this approach to its conclusion for future work.

In this paper we instead proceed by the maximum principle, inspired by \cite{BDG}. The problem of constructing entire graphical minimizers in $\mathbb{R}^{n+1}$ is closely related to the existence of singular minimizers in $\mathbb{R}^n$. This was certainly recognized in \cite{BDG}, where it was shown that $C_{33}$ is area-minimizing by constructing foliations of each side of the cone by smooth area minimizers. It is clear that the level sets of the entire graphical minimizers in \cite{BDG} resemble the leaves in this foliation, but no explicit connection is made between the results. Our approach in this paper is to make this connection explicit. As a consequence we are able to construct, in the general anisotropic case, a variety of examples with many different growth rates in the optimal dimension $n = 4$, and to recover the known examples from \cite{BDG}. 

\begin{rem}
A strategy that is similar in spirit is pursued for the area functional in \cite{CoM} and \cite{BeM}, where the question of area minimality for Lawson's cones (appropriate affine transformations of $C_{kl}$) is studied with the help of computer calculations.
\end{rem}

Our starting point is our recent work \cite{MY} in which we prove that the cones $C_{kl}$ minimize parametric elliptic functionals for all $k,\, l \geq 1$ by constructing foliations. It was known previously that minimizers of (\ref{PEF}) are smooth in the case $n = 2$ by deep work of Almgren-Schoen-Simon \cite{ASS}. Morgan \cite{Mor} later proved that minimizers are not necessarily smooth in dimension $n = 3$. Indeed, he proved that $C_{11}$ minimizes a parametric elliptic functional, by constructing a calibration. Although the foliation approach in \cite{MY} is more involved, advantages are that it removes some of the guesswork involved in constructing a calibration, and that the leaves in the foliation give a hint as to how to proceed in the anisotropic Bernstein problem. We showed in particular in \cite{MY} that for any $\mu \in (0,\,1/2)$, there is an integrand $\overline{\Psi}$ such that the sides of $C_{11}$ are foliated by minimizers of $A_{\overline{\Psi}}$ that closely resemble the level surfaces of locally Lipschitz functions that are homogeneous of degree $1 + \mu$. This involved the careful analysis of a nonlinear second-order ODE. 

The first step in this paper is to deepen the analysis of the nonlinear ODE. By studying its linearization around a solution, we obtain small perturbations of the leaves in our foliation of each side of $C_{11}$ that have the same asymptotic behavior as before, but have strictly positive or negative anisotropic mean curvature. These leaves then define locally Lipschitz functions $\overline{w}$ and $\underline{w}$ that are homogeneous of degree $1 + \mu$, constant on the leaves, vanish on $C_{11}$, and by virtue of the curvature of their level sets serve as good model candidates for super- and sub-solutions.

The second step in this work is to make a choice of integrand $\Phi$ on $\mathbb{R}^5$. Our choice agrees exactly with $\overline{\Psi}$ on $\{x_5 = 0\}$, and can be viewed as way of smoothly extending $\overline{\Psi}$ to $\mathbb{S}^4 \backslash \{x_5 = 0\}$. The case of the area functional suggests taking 
$$\Phi|_{\{x_5 = 1\}} = \left(1 + \overline{\Psi}^2\right)^{1/2},$$
and our integrand indeed resembles this choice.

The final step in this work is to ``re-stack" the level sets of the functions $\overline{w}$ and $\underline{w}$ in a way that they become legitimate super- and sub-solutions to the equation of minimal surface type defined by $\Phi$ on one side of $C_{11}$. This is accomplished by composing $\overline{w}$ and $\underline{w}$ with appropriate concave, resp. convex one-variable functions. We can then proceed as in \cite{BDG}, using these super- and sub-solutions to trap the exact solutions to the equation we wish to solve in $B_{R}$ with appropriate boundary data, and taking $R$ to infinity.

We conclude the introduction with several remarks. The first is that our approach works equally well to construct examples asymptotic to $C_{kk} \times \mathbb{R}$ for all $k \geq 1$. We focus on the case $k = 1$ for simplicity of notation and to emphasize ideas. In a later section we indicate how to generalize to higher dimensions, and we also make explicit how our approach recovers the examples from \cite{BDG}, whose construction seems at first somewhat ad-hoc. The second is that our approach does not work when $k \neq l$, because the argument relies crucially on the odd symmetry over $C_{kk}$ of solutions to the PDE associated to $\Phi$. We intend to pursue this question in future work. In \cite{Si1} Simon constructs entire minimal graphs that are asymptotic to the cylinders over a variety of area-minimizing cones with isolated singularities (in particular, all of the area-minimizing Lawson cones, which are affine transformations of $C_{kl}$ with $k+l \geq 7$ or $k+l = 6$ and $\text{min}\{k,\,l\} \geq 2$, see e.g. Chapter $5$ in \cite{La}). The existence of foliations plays an important role in that paper as well, and we believe that a combination of the ideas in that work and ours may bring further clarity to the picture. Finally, we remark that when $\Phi$ is close to the area functional on $\mathbb{S}^n$ in a strong topology, e.g. $C^4$, the results are the same as in the area case. For example, the Bernstein theorem holds up to dimension $n = 7$ \cite{Si2}, regularity of minimizers holds up to dimension $n = 6$ \cite{ASS}, and stable critical points in low dimensions are flat (\cite{CL1}, \cite{CL2}). We thus know that for our examples, the integrands are necessarily far from area. It would be interesting to weaken the topology required for such results e.g. to closeness in $C^2$ (which suffices for proving the flatness of stable critical points in dimension $n = 2$ \cite{Lin}), and our examples may shed light on this question. These remarks are discussed at greater length in a later section.

The paper is organized as follows. In Section \ref{ODEAnalysis} we recall our results from \cite{MY} concerning the foliation of each side of $C_{11}$ by minimizers of parametric elliptic functionals, and we refine our analysis of a certain nonlinear ODE to obtain perturbed foliations by hypersurfaces with anisotropic mean curvature of a desired sign. In Section \ref{IntegrandChoice} we make our choice of integrand $\Phi$, which fixes the equation we wish to solve. In Section \ref{SuperSub} we construct super- and sub-solutions to this equation using the perturbed foliations. In Section \ref{Proof} we put it all together to prove Theorem \ref{Main}. Finally, in Section \ref{DiscussionSec} we discuss generalizations of our constructions to higher dimensions, the relation to the case of the area functional, and future work.

\section*{Acknowledgments}

The authors gratefully acknowledge the support of NSF grant DMS-1854788, NSF CAREER grant DMS-2143668, an Alfred P. Sloan Research Fellowship, and a UC Irvine Chancellor's Fellowship.

\section{Foliation}\label{ODEAnalysis}
In this section we first recall for the reader's convenience the construction of foliations of each side of $C_{11}$ by minimizers of (\ref{PEF}), accomplished in \cite{MY}. We then refine the analysis from \cite{MY}, and we use this along with a study of a linearized problem to perturb the leaves of the foliation so that they have positive or negative anisotropic mean curvature. Finally, we use these leaves to define homogeneous functions on $\mathbb{R}^4$ that will serve as models for super- and sub-solutions to the PDE we eventually wish to solve.

\subsection{Previous Results}
In \cite{MY} we constructed uniformly elliptic integrands $\overline{\Psi}$ on $\mathbb{R}^{4}$ defined by analytic, even, one-variable functions $\phi$ as follows:
\begin{equation}
\overline{\Psi}(x,\,y) = \varphi(|x|,\,|y|) = \begin{cases} 
|y|\phi(|x|/|y|),\, |y| > 0, \\
 |x|\phi(|y|/|x|),\, |x| > 0.
\end{cases}
\end{equation}
We showed that for any $\mu \in (0,\,1/2)$, the function $\phi$ could be chosen such that there exists a critical point $\Sigma \subset \{|y| > |x|\}$ of $A_{\overline{\Psi}}$ of the form
$$\Sigma = \{|y| = \sigma(|x|)\},$$
where the function $\sigma$ is even, analytic, locally uniformly convex, larger than $|.|$, $\sigma(0) = 1$, and
$$\sigma(\tau) = \tau + a \tau^{-\mu} + o(\tau^{-\mu})$$
for some $a > 0$ as $\tau \rightarrow \infty$ (see Lemma $2.1$ and the discussion after Remark $8$ in \cite{MY}). Furthermore, the curve $\Gamma$ parameterized for $\tau > 0$ by $(\tau^{-1}\sigma(\tau),\, \sigma'(\tau))$ tends as $\tau \rightarrow \infty$ to $(1,\,1)$, and $\Gamma$ is bounded in a certain region of the $(w,\,z)$ plane:
\begin{equation}\label{Trapping}
\Gamma \subset \{0 < z < 1\} \cap \{z > 3/2 - w/2\}
\end{equation}
(see the proof of Lemma $2.1$ in \cite{MY}). We remark that the smoothness of $\overline{\Psi}$ away from $0$ ensures that
\begin{equation}\label{phicompat}
2\phi'(1) = \phi(1),
\end{equation}
and that $\mu$ and $\phi$ are related through the identiy
\begin{equation}\label{mudef}
\frac{\phi(1)}{2\phi''(1)} = \mu(1-\mu)
\end{equation}
(identity $(11)$ in the statement of Lemma $2.1$ from \cite{MY}).

The dilations of $\Sigma$, along with their reflections over $C_{11}$, give a foliation of each side of $C_{11}$ by minimizers of $A_{\overline{\Psi}}$.
The condition that $\Sigma$ is a critical point is equivalent to $\sigma$ solving the ODE
\begin{equation}\label{ODE}
G(\sigma)(\tau) := \sigma''(\tau) + \frac{1}{\tau}P(\sigma'(\tau)) + \frac{1}{\sigma}Q(\sigma'(\tau)) = 0,
\end{equation}
where
\begin{equation}\label{Coeffs}
P(s) = \frac{\phi'(s)}{\phi''(s)}, \quad Q(s) = \frac{s\phi'(s)-\phi(s)}{\phi''(s)},
\end{equation}
and the results were obtained through the analysis of this ODE.

\begin{rem}
In \cite{MY} we defined $\overline{\Psi}$ in $\{|x| > 0\}$ and $\{|y| > 0\}$ by two different functions $\phi$ and $\psi$. However, the conditions in \cite{MY} that $\phi$ and $\psi$ are required to satisfy (namely, $(7)$ and $(8)$ in Lemma $2.1$ from \cite{MY}), are invariant under taking convex combinations when $k = l$, so we can assume $\phi = \psi$ after replacing $\overline{\Psi}(x,\,y)$ by $(\overline{\Psi}(x,\,y) + \overline{\Psi}(y,\,x))/2$.
\end{rem}


\subsection{Refinement of Foliation Analysis}
From hereon out, we fix a choice of $\mu \in (0,\,1/2)$ and $\phi$ as in the previous subsection. For the purposes of this paper we need to make the asymptotic expansion of $\sigma$ a little more precise. 

We first establish some notation. Here and for the rest of the paper we will let $c$ denote a small positive constant that depends only on $\phi$, and its value may change from line to line. For a smooth function $h$ on $(0,\,\infty)$, we denote by $O_2(h)$ a smooth function on $(0,\,\infty)$ such that, for all $\tau \geq 1$, 
$$|O_2(h)(\tau)| \leq c^{-1}|h(\tau)|, \quad |O_2(h)'(\tau)| \leq c^{-1}|h'(\tau)|,\, \quad |O_2(h)''(\tau)| \leq c^{-1}|h''(\tau)|.$$

\begin{prop}\label{Asymptotics}
There exists $a > 0$ and $b \in \mathbb{R}$ such that
$$\sigma = \tau + a\tau^{-\mu} + b\tau^{\mu-1} + O_2(\tau^{-1-2\mu}).$$
\end{prop}

\noindent For our purposes, the following corollary in fact suffices:
\begin{equation}\label{Asymptotic}
\sigma = \tau + a\tau^{-\mu} + O_2(\tau^{-1/2}) \text{ for some } a > 0.
\end{equation}

Before proving Proposition \ref{Asymptotics}, we recall a few things from the proof of the results in the previous subsection (more precisely, the proof of Lemma $2.1$ from \cite{MY}). First, the ODE (\ref{ODE}) for $\sigma$ can be written as a first-order autonomous system for ${\bf X}(s) := (e^{-s}\sigma(e^s),\, \sigma'(e^s))$ of the form
$${\bf X}'(s) = {\bf V}({\bf X}(s)),$$
where
$${\bf V}(w,\,z) = (-w + z,\, -P(z) - Q(z)/w).$$
Using the identities (\ref{phicompat}) and (\ref{mudef}) we have that ${\bf V}(1,\,1) = (0,\,0)$ and that
$$D{\bf V}(1,\,1) = M := \begin{pmatrix}
-1 & 1 \\
-\mu(1-\mu) & -2
\end{pmatrix}.
$$
The matrix $M$ has eigenvalues $-1-\mu$ and $\mu - 2$, corresponding to eigenvectors in the directions of lines with the slopes $-\mu$ and $\mu-1$. We proved in Lemma $2.1$ of \cite{MY} that the solution curve ${\bf X}$ tends to $(1,\,1)$ as $s$ tends to infinity, and by (\ref{Trapping}) is trapped in a region which excludes the line through $(1,\,1)$ with slope $\mu - 1$.

\begin{proof}[{\bf Proof of Proposition \ref{Asymptotics}:}]
Let 
$${\bf Y}(s) = {\bf X}(s) - (1,\,1), \quad {\bf W}(w,\,z) = {\bf V}(w+1,\,z+1),$$
so that
$${\bf Y}' = {\bf W}({\bf Y}), \quad {\bf W}(0) = 0, \quad D{\bf W}(0) = M.$$
The eigenvalues of $M^T + M$ are $-3 \pm \sqrt{1+(1-\mu(1-\mu))^2} \in (-5,\, -3/2)$. Since $|{\bf Y}|$ tends to zero, we conclude from the ODE for $|{\bf Y}|^2$ that
\begin{equation}\label{ModBounds}
ce^{-5s} \leq |{\bf Y}|^2(s) \leq c^{-1}e^{-\frac{3}{2}s}, \quad s \geq 0.
\end{equation}
To obtain (\ref{ModBounds}) we use that
\begin{equation}\label{ExpandedODE}
|{\bf Y}' - M{\bf Y}| \leq c^{-1}|{\bf Y}|^2, \quad s \geq 0.
\end{equation}
Letting 
$${\bf p} = (1,\,-\mu), \quad {\bf q} = (1,\,\mu-1)$$ 
denote the eigenvectors of $M$, decomposing ${\bf Y}$ into a linear combination of these vectors, and using (\ref{ModBounds}) in (\ref{ExpandedODE}) gives in a similar way that for some $a_0 \in \mathbb{R}$,
$$|{\bf Y}(s) - a_0e^{-(1+\mu)s}{\bf p}| \leq c^{-1}e^{-3s/2}, \quad s \geq 0.$$
We conclude from this the following improvement of (\ref{ModBounds}):
$$|{\bf Y}|^2 \leq c^{-1}e^{-2(1+\mu)s}, \quad s \geq 0.$$
Using this improved estimate in (\ref{ExpandedODE}) we conclude that for some $a,\,b \in \mathbb{R}$,
$$|{\bf Y}(s) - ae^{-(1+\mu)s}{\bf p} - be^{(\mu-2)s}{\bf q}| \leq c^{-1}e^{-2(1+\mu)s}, \quad s \geq 0.$$
The conclusion follows from this expression and from the second component of the ODE for ${\bf Y}$, provided $a > 0$. 
To show that $a > 0$, note first that by (\ref{Trapping}), we have $a \geq 0$ and that if $a = 0$ then $b = 0$. In the latter case we would have $|{\bf Y}|^2 \leq c^{-1}e^{-4(1+\mu)s},\, s \geq 0$, and upon examining the ODE (\ref{ExpandedODE}) for ${\bf Y}$ again we could improve the expansion of ${\bf Y}$ to
$$|{\bf Y}(s) - a_1e^{-(1+\mu)s}{\bf p} - b_1e^{(\mu-2)s}{\bf q}| \leq c^{-1}e^{-4(1+\mu)s}, \quad s \geq 0.$$
If $a_1 \neq 0$ then the upper bound $|{\bf Y}|^2 \leq c^{-1}e^{-4(1+\mu)s},\, s \geq 0$ is violated, hence $a_1 = 0$, and again by (\ref{Trapping}) we have $b_1 = 0$. We conclude that $|{\bf Y}|^2 \leq c^{-1}e^{-8(1+\mu)s},\, s \geq 0$, which contradicts the lower bound in (\ref{ModBounds}).
\end{proof}

\subsection{Linear ODE}
Our goal now is to perturb the leaves in the foliation determined by the function $\sigma$. We let $L$ denote the linearization of the ODE (\ref{ODE}) at $\sigma$, namely,
\begin{align*}
Lf &= f'' + \left(\frac{1}{\tau} + \frac{\sigma'}{\sigma} + \frac{\sigma''\phi'''}{\phi''}\right)f' + \frac{\phi - \sigma'\phi'}{\sigma^2\phi''}f \\
&:= f'' + [\log(p)]'f' + qf,
\end{align*}
where 
$$p(s) = s\sigma(s)\phi''(\sigma'(s)).$$
Using the invariance of the operator $G$ under the Lipschitz rescalings
$$\sigma \rightarrow \sigma_{\lambda} := \lambda^{-1}\sigma(\lambda \tau),$$
we see that the function
$$f_0(\tau) := -\frac{d}{d\lambda} \sigma_{\lambda}(\tau)|_{\lambda = 1} = \sigma(\tau) - \tau\sigma'(\tau)$$
solves the linearized equation
$$Lf_0 = 0.$$ 
Furthermore, by the properties of $\sigma$ and the estimate (\ref{Asymptotic}), $f_0$ is smooth, positive, even, and satisfies
\begin{equation}\label{f0Asymptotics}
f_0 = a(1+\mu)\tau^{-\mu} + O_2(\tau^{-1/2}).
\end{equation}
For $g$ continuous, the solution to the ODE
$$Lf = g, \quad f(0) = f'(0) = 0$$
can be written
\begin{equation}\label{LinearODESolution}
f(\tau) = f_0(\tau) \int_0^{\tau} \frac{1}{f_0^2(t)p(t)}\int_0^{t} g(s)p(s)f_0(s)\,ds\,dt.
\end{equation}
Taking
$$g = \sigma^{-5/2}$$
in (\ref{LinearODESolution}) we obtain a smooth even solution $f_1$. By the asymptotics (\ref{f0Asymptotics}) of $f_0$ and the formula (\ref{LinearODESolution}) for $f_1$, we have for some $d \in \mathbb{R}$ that
\begin{equation}\label{f1Asymptotics}
f_1 = d\tau^{-\mu} + O_2(\tau^{-1/2}).
\end{equation}
Below we will often bound quantities by powers of $\sigma$, which serves as a strictly positive regularization of the function $|\tau|$. Using (\ref{Asymptotic}) and (\ref{f1Asymptotics}) and applying Taylor's theorem with remainder to $P$ and $Q$, it is straightforward to show for $\epsilon$ small that
\begin{equation}\label{Expansion}
|G(\sigma + \epsilon f_1) - \epsilon L(f_1)| \leq c^{-1}\epsilon^2\sigma^{-3-2\mu}.
\end{equation}
Indeed, Taylor expansion gives for $|\epsilon| > 0$ small and any $\tau > 0$ that
\begin{align*}
\epsilon^{-2}|G(\sigma + \epsilon f_1) - \epsilon L(f_1)| &\leq \frac{\|P''\|_{L^{\infty}([0,\,\sigma'(\tau) + |\epsilon f_1'(\tau)|])}}{\tau}f_1'^2 \\
&+ \|Q\|_{C^2([0,\,1])}\left(\frac{f_1'^2}{\sigma} + \frac{|f_1f_1'|}{\sigma^2} + \frac{f_1^2}{\sigma^3}\right) \\
&+ |\epsilon| \|Q\|_{C^2([0,\,1])}\left(\frac{|f_1f_1'^2|}{\sigma^2} + \frac{|f_1^2f_1'|}{\sigma^3}\right) \\
&+ \epsilon^2\|Q\|_{C^2([0,\,1])}\frac{f_1^2f_1'^2}{\sigma^3}.
\end{align*}
Using that $P$ is odd and $\sigma$ and $f_1$ are even we have that
$$\frac{\|P''\|_{L^{\infty}([0,\,\sigma'(\tau) + |\epsilon f_1'(\tau)|])}}{\tau} \leq c^{-1}\frac{1}{\sigma},$$
and the remaining terms can be bounded using (\ref{Asymptotic}) and (\ref{f1Asymptotics}).

In summary, we have proven:
\begin{prop}\label{ODESuperSub}
There exists $\epsilon_0 > 0$ small such that the functions
$$\overline{\sigma} := \sigma + \epsilon_0 f_1, \quad \underline{\sigma} := \sigma - \epsilon_0 f_1$$
are even, locally uniformly convex, larger than $|\tau|$, and have the asymptotics
\begin{equation}\label{SuperSubAsymptotics}
\overline{\sigma} = \tau + \overline{a} \tau^{-\mu} + O_2(\tau^{-1/2}), \quad \underline{\sigma} = \tau + \underline{a} \tau^{-\mu} + O_2(\tau^{-1/2})
\end{equation}
for some $\overline{a},\,\underline{a} > 0$. Furthermore, they satisfy for all $\tau \in \mathbb{R}$ that
\begin{equation}\label{ODERHS}
G(\overline{\sigma}) \geq c\sigma^{-5/2}, \quad -G(\underline{\sigma}) \geq c\sigma^{-5/2}
\end{equation}
and
\begin{equation}\label{Ordering}
\frac{1}{2} \overline{\sigma}(2\tau) \leq \underline{\sigma}(\tau) \leq \overline{\sigma}(\tau).
\end{equation}
\end{prop}
\noindent The inequality (\ref{Ordering}) follows from the asymptotics (\ref{Asymptotic}) for $\sigma$ and (\ref{f1Asymptotics}) for $f_1$, and the smallness of $\epsilon_0$. For (\ref{ODERHS}) we use (\ref{Expansion}) and that $5/2 < 3+2\mu$. 

Geometrically, the surface 
$$\overline{\Sigma} := \{|y| = \overline{\sigma}(|x|)\}$$ 
has anisotropic mean curvature vector pointing ``away from" $C_{11}$, and the opposite is true for the hypersurface
$$\underline{\Sigma} := \{|y| = \underline{\sigma}(|x|)\}.$$

\begin{rem}\label{Flexibility}
The choice $g = \sigma^{-5/2}$ is convenient but there is flexibility in the choice of exponent. For all arguments in this paper, the choice $g = \sigma^{-\beta-2}$ with
$$\mu < \beta < 1-\mu$$
would suffice. These inequalities guarantee that
$$\sigma = \tau + a \tau^{-\mu} + O_2(\tau^{-\beta}),$$ 
that the solution $f$ to $Lf = g,\, f(0) = f'(0) = 0$ satisfies
$$f = d\tau^{-\mu} + O_2(\tau^{-\beta})$$
for some $d \in \mathbb{R}$, and that $G(\sigma + \epsilon f) \geq cg$ for $\epsilon > 0$ small. The arguments below can also be treated in a similar way as presented for such choices of $\beta$.
\end{rem}


\subsection{Model Super- and Sub-solutions}
To conclude the section we define functions that are homogeneous of degree $1 + \mu$ and take the value $1$ on $\overline{\Sigma}$ and $\underline{\Sigma}$. These serve as model super- and sub-solutions to the PDE we eventually wish to solve. First, we define functions $\overline{w_0},\, \underline{w_0}$ of two variables $(\xi,\, \zeta)$ in $\{\zeta > |\xi|\}$ by
$$\overline{w_0}(\lambda^{-1}\tau,\,\lambda^{-1}\overline{\sigma}(\tau)) = \underline{w_0}(\lambda^{-1}\tau,\,\lambda^{-1}\underline{\sigma}(\tau)) = \lambda^{-1-\mu}.$$ 
Here $\lambda > 0$ and $\tau \in \mathbb{R}$. We extend these functions over the diagonals by odd reflection, and we define them to vanish on the diagonals. The continuity of $\overline{w_0}$ follows from the identity $\overline{w_0}(\tau/\overline{\sigma}(\tau),\,1) = \overline{\sigma}^{-1-\mu}$ and taking $\tau$ to $\pm \infty$, and similarly for $\underline{w_0}$. The calculations below show that $\overline{w_0},\, \underline{w_0}$ are in fact locally Lipschitz.

Here and below we evaluate the quantities of interest for $\overline{w_0}$ at the point $\lambda^{-1}(\tau,\, \overline{\sigma}(\tau))$, and for $\underline{w_0}$ at $\lambda^{-1}(\tau,\, \underline{\sigma}(\tau))$. We calculate
\begin{equation}\label{GradExpression}
\nabla \overline{w_0} = (1+\mu)\frac{\lambda^{-\mu}}{\overline{\sigma} - \tau \overline{\sigma}'}(-\overline{\sigma}',\,1), \quad \nabla \underline{w_0} = (1+\mu)\frac{\lambda^{-\mu}}{\underline{\sigma} - \tau \underline{\sigma}'}(-\underline{\sigma}',\,1).
\end{equation}
Using Proposition \ref{ODESuperSub} we conclude that
\begin{equation}\label{GradEst0}
c\lambda^{-\mu}\sigma^{\mu} \leq |\nabla \overline{w_0}|,\, |\nabla \underline{w_0}| \leq c^{-1}\lambda^{-\mu}\sigma^{\mu}.
\end{equation}
It follows that $c \leq |\nabla \overline{w_0}(\tau/\overline{\sigma}(\tau),\,1)|,\, |\nabla \underline{w_0}(\tau/\underline{\sigma}(\tau),\,1)| \leq c^{-1}$ for all $\tau \in \mathbb{R}$, which establishes the local Lipschitz regularity of these functions.

Next we calculate
\begin{equation}\label{HessianFormula}
\begin{split}
D^2\overline{w_0} = &(1+\mu)\frac{\lambda^{1-\mu}}{(\overline{\sigma} - \tau\overline{\sigma}')^3} \cdot \\
&\left[\mu(\overline{\sigma}-\tau\overline{\sigma}')(-\overline{\sigma}',\,1) \otimes (-\overline{\sigma}',\,1) - \overline{\sigma}''(-\overline{\sigma},\,\tau) \otimes (-\overline{\sigma},\,\tau)\right],
\end{split}
\end{equation}
and the same expression with $\overline{\sigma}$ replaced by $\underline{\sigma}$ for $D^2\underline{w_0}$. By Proposition \ref{ODESuperSub}, each of the four entries in the matrix on the second line of (\ref{HessianFormula}) is bounded in absolute value by $c^{-1}\sigma^{-1/2}$, which gives
\begin{equation}\label{HessEst0}
|D^2\overline{w_0}|,\, |D^2\underline{w_0}| \leq c^{-1}\lambda^{1-\mu}\sigma^{3\mu-1/2}.
\end{equation}

For $x,\,y \in \mathbb{R}^2$ we let
$$\overline{w}(x,\,y) = \overline{w_0}(|x|,\,|y|), \quad \underline{w}(x,\,y) = \underline{w_0}(|x|,\,|y|).$$
We evaluate the quantities of interest for $\overline{w}$, resp. $\underline{w}$ on $\{(|x|,\,|y|) = \lambda^{-1}(\tau,\,\overline{\sigma}(\tau)),\, \tau \geq 0,\, \lambda > 0\}$, resp. $\{(|x|,\,|y|) = \lambda^{-1}(\tau,\,\underline{\sigma}(\tau)),\, \tau \geq 0,\, \lambda > 0\}$.
By (\ref{GradEst0}) we have
\begin{equation}\label{GradEst}
c\lambda^{-\mu}\sigma^{\mu} \leq |\nabla \overline{w}|,\, |\nabla \underline{w}| \leq c^{-1}\lambda^{-\mu}\sigma^{\mu}. 
\end{equation}
We now bound the Hessian. In coordinates that are polar in $x$ and $y$ separately, the matrix $D^2\overline{w}$ contains six nonzero entries, four of which coincide with entries of $D^2\overline{w_0}$ at $\lambda^{-1}(\tau,\,\overline{\sigma}(\tau))$. For $\tau > 0$, the remaining two (corresponding to pure second derivatives in $x$ and $y$ in directions tangent to circles in these variables centered at the origin) are $\lambda \tau^{-1}\partial_{\xi}\overline{w_0}$ and $\lambda\overline{\sigma}^{-1}\partial_{\zeta}\overline{w_0}$. Using (\ref{GradExpression}) we have the bounds
$$\lambda \tau^{-1}|\partial_{\xi}\overline{w_0}|,\, \lambda\overline{\sigma}^{-1}|\partial_{\zeta}\overline{w_0}| \leq c^{-1} \lambda^{1-\mu}\sigma^{\mu-1} \leq c^{-1}\lambda^{1-\mu}\sigma^{3\mu-1/2},$$
and similarly for $D^2\underline{w}$. We conclude from (\ref{HessEst0}) that
\begin{equation}\label{HessEst}
|D^2\overline{w}|,\, |D^2\underline{w}| \leq c^{-1}\lambda^{1-\mu}\sigma^{3\mu-1/2}.
\end{equation}

Finally, we relate the curvature of the level sets of $\overline{w},\, \underline{w}$ to a PDE. Recall that
$$\overline{\Psi}(x,\,y) = \varphi(|x|,\,|y|) = |y|\phi(|x|/|y|) = |x|\phi(|y|/|x|).$$
Using that $\partial_{\xi} \overline{w_0} < 0$ in $\{\zeta > \xi > 0\}$, we have in $\{|y| > |x|\}$ that
\begin{equation}\label{HomogeneousPDE1}
\overline{\Psi}_{ij}(\nabla \overline{w})\overline{w}_{ij} = \varphi_{11}\overline{w_0}_{\xi\xi} - 2\varphi_{12}\overline{w_0}_{\xi\zeta} + \varphi_{22}\overline{w_0}_{\zeta\zeta} - \frac{\varphi_1}{|x|} + \frac{\varphi_2}{|y|},
\end{equation}
where the derivatives of $\varphi$ are evaluated at $(|\partial_{\xi} \overline{w_0}|,\,|\partial_{\zeta}\overline{w_0}|)$. Using the formulae (\ref{GradExpression}) and (\ref{HessianFormula}) and the relation between $\varphi$ and $\phi$, and evaluating the above expression on $\{(|x|,\,|y|) = \lambda^{-1}(\tau,\,\overline{\sigma}(\tau)),\, \tau > 0,\, \lambda > 0\},$ we get
$$\overline{\Psi}_{ij}(\nabla \overline{w})\overline{w}_{ij} = -\lambda\phi''(\overline{\sigma}')G(\overline{\sigma}).$$
Indeed, up to the factor $-\lambda\phi''(\sigma')$, the first three terms on the right side of (\ref{HomogeneousPDE1}) contribute the first term in the expression (\ref{ODE}) for $G$, and the last two contribute the second two terms in the expression for $G$. The analogous calculations hold for $\underline{w}$, with $\overline{\sigma}$ replaced by $\underline{\sigma}$. Combining this calculation with Proposition \ref{ODESuperSub} we obtain:

\begin{prop}\label{CurvatureTerm}
We have
$$\overline{\Psi}_{ij}(\nabla \overline{w})\overline{w}_{ij} = -\lambda\phi''(\overline{\sigma}')G(\overline{\sigma}) \leq -c\lambda\sigma^{-5/2}$$
on $\{(|x|,\,|y|) = \lambda^{-1}(\tau,\,\overline{\sigma}(\tau)),\, \tau \geq 0,\,\lambda > 0\}$, and
$$\overline{\Psi}_{ij}(\nabla \underline{w})\underline{w}_{ij} = -\lambda\phi''(\underline{\sigma}')G(\underline{\sigma}) \geq c\lambda\sigma^{-5/2}$$
on $\{(|x|,\,|y|) = \lambda^{-1}(\tau,\,\underline{\sigma}(\tau)),\, \tau \geq 0,\,\lambda > 0\}$.
\end{prop}
To conclude the section we note that inequality (\ref{Ordering}) from Proposition \ref{ODESuperSub} and homogeneity imply that
\begin{equation}\label{Ordering2}
\underline{w}(\cdot) \leq \overline{w}(2\cdot) = 2^{1+\mu}\overline{w}(\cdot) \text{ in } \{|y| > |x|\}.
\end{equation}

\section{Choice of Integrand}\label{IntegrandChoice}
In this section we fix our choice of integrand. Let $\overline{\Psi}$ be the same as above, and let $F$ be a smooth even convex function on $\mathbb{R}$ such that
$$F(s) = s + \frac{1}{8}s^{-1},\, s > 1/2.$$
We let $\Psi$ be any smooth locally uniformly convex function on $\mathbb{R}^4$, depending only on $|x|$ and $|y|$ and invariant under exchanging $x$ and $y$, such that
$$\Psi = F(\overline\Psi) \text{ in } \{\overline{\Psi} > 1\}.$$
(See Remark \ref{IntegrandExtension}). Finally, we define
\begin{equation}\label{Phi}
\Phi(x,\,y,\,z) := \begin{cases} |z| \Psi\left(\frac{x}{|z|},\, \frac{y}{|z|}\right), \quad z \in \mathbb{R} \backslash \{0\}, \\
\overline{\Psi}(x,\,y), \quad z = 0.
\end{cases}
\end{equation}
\begin{prop}\label{UniformEllipticity}
The function $\Phi$ is in $C^{\infty}(\mathbb{S}^4)$ and is a uniformly elliptic integrand.
\end{prop}
\begin{proof}
On $\{|z| > 0\}$ this follows from the local uniform convexity of $\Psi$, which in $\{\overline{\Psi} > 1\}$ follows from the identity
$$D^2\Psi = F'(\overline{\Psi})D^2\overline{\Psi} + F''(\overline{\Psi})\nabla \overline{\Psi} \otimes \nabla \overline{\Psi},$$
the uniform ellipticity of $\overline{\Psi}$ and the fact that $F''(s) > 0$ for $s > 1$. We now examine the points on $\mathbb{S}^4 \cap \{z = 0\}$. Using the fact that $F(s) = s + s^{-1}/8$ for $s > 1/2$, we have in a neighborhood of $\mathbb{S}^4 \cap \{z = 0\}$ that
$$\Phi(x,\,y,\,z) = \overline{\Psi}(x,\,y) + \frac{z^2}{8\overline{\Psi}(x,\,y)}.$$
Thus, on a hyperplane tangent to $\mathbb{S}^4$ on $\{z = 0\}$, the horizontal second derivatives of $\Phi$ at the point of tangency agree with those of $\overline{\Psi}$, the mixed horizontal and vertical second derivatives vanish, and the pure vertical second derivative is strictly positive, completing the proof.
\end{proof}

\begin{rem}\label{IntegrandExtension}
This can be accomplished e.g. by taking $F$ to be a positive constant in a small neighborhood of $0$ and locally uniformly convex otherwise, and defining $\Psi$ to be the sum of $F(\overline{\Psi})$ and a small multiple of an appropriate radial cutoff of $|x|^2 + |y|^2$.
\end{rem}

\section{Super and Sub Solutions}\label{SuperSub}
Let $\Phi,\,\Psi$ be as in the previous section. We note as in \cite{M} that the graph of a function $u$ on a domain $\Omega \subset \mathbb{R}^4$ is a critical point of $A_{\Phi}$ if and only if
\begin{equation}\label{GraphEquation}
\Psi_{ij}(\nabla u)u_{ij} = 0
\end{equation}
in $\Omega$. In this section we ``re-stack" the level sets of $\overline{w}$ and $\underline{w}$ to obtain super- and sub-solutions of (\ref{GraphEquation}) in $\{|y| > |x|\}$, without changing their growth rates.

\subsection{Supersolution}
To begin we fix the quantities
$$\gamma := \frac{2\mu}{1+\mu},\, \quad M := \frac{2}{1/2 - \mu}, \quad \delta := \frac{\gamma}{M}.$$
We define
$$\overline{u} = \overline{H}(\overline{w}),$$
where $\overline{H}(0) = 0$ and
$$\overline{H}'(s) = A|s|^{-\gamma} + e^{B\int_{|s|}^{\infty} \frac{t^{\delta-1}}{1+t^{2\delta}}\,dt} = A|s|^{-\gamma} + e^{BI(|s|)}.$$
We note that $\overline{H}$ is well-defined because $\gamma < 1$. We claim for $B = A^2$ and $A$ sufficiently large that $\overline{u}$ is a super-solution to (\ref{GraphEquation}) in $\{|y| > |x|\}$.

To see this, note first that
$$\overline{H}'(s) \geq 1 + As^{-\frac{2\mu}{1+\mu}}.$$
Combining this with the estimate (\ref{GradEst}) we conclude on $\{(|x|,\,|y|) = \lambda^{-1}(\tau,\,\overline{\sigma}(\tau)),\, \tau \geq 0,\, \lambda > 0\}$ that
$$|\nabla \overline{u}| =  \overline{H}'(\lambda^{-1-\mu})|\nabla \overline{w}| \geq c(A \lambda^{\mu} + \lambda^{-\mu}) \geq cA^{1/2}.$$
For $A > 1$ sufficiently large we conclude that
$$\nabla \overline{u} (\{|y| > |x|\}) \subset \{\overline{\Psi} > 1\},$$
thus $\nabla \overline{u}$ always lies in the region where $\Psi = F(\overline{\Psi})$. Henceforth we assume $A$ has been chosen at least this large. We calculate using the one-homogeneity of $\overline{\Psi}$ that
\begin{align*}
\Psi_{ij}(\nabla \overline{u})\overline{u}_{ij} &= F'(\overline{H}'(\overline{w})\overline{\Psi}(\nabla \overline{w}))\overline{\Psi}_{ij}(\nabla \overline{w})\overline{w}_{ij} \\
&+ F''(\overline{H}'(\overline{w})\overline{\Psi}(\nabla \overline{w}))\overline{H}''(\overline{w})\overline{\Psi}^2(\nabla \overline{w}) \\
&+ F''(\overline{H}'(\overline{w})\overline{\Psi}(\nabla \overline{w}))\overline{H}'(\overline{w})\overline{\Psi}_i(\nabla \overline{w})\overline{\Psi}_j(\nabla \overline{w})\overline{w}_{ij} \\
&:= I + II + III
\end{align*}
in $\{|y| > |x|\}$. We have that $F'(s) \geq 1/2$ and $F''(s) = s^{-3}/4$ when $s > 1$, and that $\overline{H}'' < 0$. Using this, along with Proposition \ref{CurvatureTerm}, and the estimates (\ref{GradEst}) and (\ref{HessEst}) on $\nabla \overline{w}$ and $D^2\overline{w}$, we conclude that
$$I \leq -c\lambda \sigma^{-5/2}, \, II \leq -c\left|\overline{H}''\right|\overline{H}'^{-3}\lambda^{\mu}\sigma^{-\mu},\, III \leq c^{-1}\overline{H}'^{-2}\lambda^{1+2\mu}\sigma^{-1/2}$$
on $\{(|x|,\,|y|) = \lambda^{-1}(\tau,\,\overline{\sigma}(\tau)),\, \tau \geq 0,\, \lambda > 0\}$.
Here $\overline{H}$ and its derivatives are evaluated at $\lambda^{-1-\mu}$. We conclude in particular that
$$\Psi_{ij}(\nabla \overline{u})\overline{u}_{ij} \leq -c\left(\lambda \sigma^{-5/2} + \left|\overline{H}''\right|\overline{H}'^{-3}\lambda^{\mu}\sigma^{-\mu}\right) + c^{-1}\overline{H}'^{-2}\lambda^{1+2\mu}\sigma^{-1/2}.$$
We can make sure this is nonpositive provided
$$c^{-2} \leq \overline{H}'^2(\lambda^{-1-\mu})\lambda^{-2\mu}\sigma^{-2} + \lambda^{-1-\mu}\left|\overline{H}''\right|/\overline{H}'(\lambda^{-1-\mu})\sigma^{1/2 - \mu}.$$
After the change of variable $s = \lambda^{-1-\mu},\, \tilde{\sigma} = \sigma^{1/2-\mu}$ this desired inequality becomes
\begin{equation}\label{DesiredSuper}
c^{-2} \leq s|\overline{H}''|/\overline{H}'(s)\tilde{\sigma} + \overline{H}'^2(s)s^{\gamma}\tilde{\sigma}^{-M} := E
\end{equation}
for all $s > 0$ and $\tilde{\sigma} \geq 1$. We have
$$s|\overline{H}''|/\overline{H}'(s) = \frac{A\gamma s^{-\gamma} + B\frac{s^{\delta}}{1 + s^{2\delta}}e^{BI}}{As^{-\gamma}+e^{BI}},$$
and since $\overline{H}' \geq 1 + As^{-\gamma}$,
$$\overline{H}'^2(s)s^{\gamma} \geq A^2s^{-\gamma} + s^{\gamma}.$$ 
To continue we split the verification of (\ref{DesiredSuper}) into two cases. In the case $s \leq 1$ we have
\begin{align*}
E &\geq \frac{A\gamma s^{-\gamma-\delta} + \frac{B}{2}e^{BI}}{As^{-\gamma} + e^{BI}}s^{\delta}\tilde{\sigma} + A^2s^{-\gamma}\tilde{\sigma}^{-M} \\
&\geq \text{min}\{\gamma,\,1/2\}X + A^2X^{-M}  \quad \text{ (here } X = s^{\delta}\tilde{\sigma}) \\
&\geq cA^{\frac{2}{M+1}}
\end{align*}
for any $B \geq 1$, hence (\ref{DesiredSuper}) is true in the case $s \leq 1$ for $A$ large and any $B \geq 1$. We may take e.g. $B = A^2$. Then in the remaining case $s \geq 1$
we have
\begin{align*}
E &\geq \frac{B\frac{s^{2\delta}}{1+s^{2\delta}}e^{BI}}{A + e^{BI}}s^{-\delta}\tilde{\sigma} + s^{\gamma}\tilde{\sigma}^{-M}  \\
&\geq \frac{1}{2}B\frac{e^{BI}}{A + e^{BI}}X + X^{-M} \quad \text{ (here } X = s^{-\delta}\tilde{\sigma}) \\
&\geq \frac{B}{2(A + 1)}X + X^{-M} \\
&\geq \frac{1}{4}AX + X^{-M} \quad \text{(using that } B = A^2 > 1) \\
&\geq cA^{\frac{M}{M+1}},
\end{align*}
concluding the proof of (\ref{DesiredSuper}) provided $B = A^2$ and $A$ is large.

\begin{rem}\label{GeomMeaning}
Roughly, the term ``$I$'' in the PDE, which represents the curvature of the level sets of $\overline{u}$, dominates in a region of $\{|y| > |x|\}$ that lies away from the boundary $C_{11}$, is tangent to $C_{11}$ near $0$, and separates sub-linearly from $C_{11}$ for $|y|$ large. The term ``$II$," which represents the remaining vertical curvature of the graph of $\overline{u}$, dominates near $C_{11}$.
\end{rem}

\subsection{Subsolution}
The subsolution is similar to and in fact easier to construct than the supersolution. Let $\gamma,\,M,$ and $\delta$ be as above. We first define
$$u_0 = \underline{H}(\underline{w}),$$
with $\underline{H}(0) = 0$ and
$$\underline{H}'(s) = e^{-C\int_{s}^{\infty} \frac{t^{\delta-1}}{1+t^{2\delta}}\,dt}.$$
We claim that for all $C \geq C_0$ large, there exists $\lambda_C > 0$ small such that the function $u_0$ is a sub-solution of (\ref{GraphEquation}) in 
$$\Omega_C := \{\underline{w} > \lambda_C^{-1-\mu}\} = \{(|x|,\,|y|) = \lambda^{-1}(\tau,\,\underline{\sigma}(\tau)),\, \tau \geq 0,\, 0 < \lambda < \lambda_C\}.$$ 
The value $\lambda_C$ is chosen for any $C \geq 1$ as follows. First, we choose $\lambda_C < 1$ small such that 
\begin{equation}\label{GradCondition}
\underline{H}'(\lambda_C^{-1-\mu}) \geq 1/2.
\end{equation}
Using (\ref{GradEst}) we conclude that
$$|\nabla u_0| \geq c\lambda_C^{-\mu} \text{ in } \Omega_C.$$
After possibly taking $\lambda_C$ smaller we thus have
$$\nabla u_0(\Omega_C) \subset  \{\overline{\Psi} > 1\}.$$
This fixes our choice of $\lambda_C$ for arbitrary $C \geq 1$.
Since $\Psi = F(\overline{\Psi})$ in $\{\overline{\Psi} > 1\}$ we conclude in $\Omega_C$ that
\begin{align*}
\Psi_{ij}(\nabla u_0)(u_0)_{ij} &= F'(\underline{H}'(\underline{w})\overline{\Psi}(\nabla \underline{w}))\overline{\Psi}_{ij}(\nabla \underline{w})\underline{w}_{ij} \\
&+ F''(\underline{H}'(\underline{w})\overline{\Psi}(\nabla \underline{w}))\underline{H}''(\underline{w})\overline{\Psi}^2(\nabla \underline{w}) \\
&+ F''(\underline{H}'(\underline{w})\overline{\Psi}(\nabla \underline{w}))\underline{H}'(\underline{w})\overline{\Psi}_i(\nabla \underline{w})\overline{\Psi}_j(\nabla \underline{w})\underline{w}_{ij} \\
&= I + II + III.
\end{align*}
Using that $F' \geq 1/2,\, F''(s) = s^{-3}/4$ for $s \geq 1$, Proposition \ref{CurvatureTerm}, the estimates (\ref{GradEst}) and (\ref{HessEst}) on $\nabla \underline{w}$ and $D^2\underline{w}$, and (\ref{GradCondition}), we conclude on $\{(|x|,\,|y|) = \lambda^{-1}(\tau,\,\underline{\sigma}(\tau)),\, \tau \geq 0,\, 0 < \lambda < \lambda_C\}$ that
$$I \geq c\lambda \sigma^{-5/2}, \, II \geq c\underline{H}''(\lambda^{-1-\mu})\lambda^{\mu}\sigma^{-\mu},\, III \geq -c^{-1}\lambda^{1+2\mu}\sigma^{-1/2}.$$
We emphasize here that $c$ does not depend on $C$. We have in particular that
$$\Psi_{ij}(\nabla u_0)(u_0)_{ij} \geq c\left(\lambda \sigma^{-5/2} + \underline{H}''(\lambda^{-1-\mu})\lambda^{\mu}\sigma^{-\mu}\right) - c^{-1}\lambda^{1+2\mu}\sigma^{-1/2}$$
in $\Omega_C$. After the same change of variable as in the previous subsection, we conclude that $u_0$ is a sub-solution of (\ref{GraphEquation}) in $\Omega_C$ provided
\begin{equation}\label{DesiredSub}
c^{-2} < s\underline{H}''(s)\tilde{\sigma} + s^{\gamma}\tilde{\sigma}^{-M} := E \text{ for all } s \geq \lambda_C^{-1-\mu},\, \tilde{\sigma} \geq 1.
\end{equation}
Since 
$$s\underline{H}''(s) = C\frac{s^{\delta}}{1+s^{2\delta}}\underline{H}'(s) \geq \frac{1}{4}Cs^{-\delta}$$ 
for $s \geq \lambda_C^{-1-\mu}$ (here we use (\ref{GradCondition}) again), we have
\begin{align*}
E &\geq \frac{1}{4}Cs^{-\delta}\tilde{\sigma} + s^{\gamma}\tilde{\sigma}^{-M} \\
&= \frac{1}{4}CX + X^{-M} \text{ (here } X = s^{-\delta}\tilde{\sigma}) \\
&\geq cC^{\frac{M}{M+1}}
\end{align*}
for $s \geq \lambda_{C}^{-1-\mu},\, \tilde{\sigma} \geq 1$. Thus, (\ref{DesiredSub}) holds for any choice $C \geq C_0$ large.

\begin{rem}\label{GeomMeaning2}
Again, the term ``$I$" representing the curvature of the level sets of $u_0$ dominates in a region of $\Omega_C$ that lies away from $C_{11}$ and separates sub-linearly from $C_{11}$ as $|y|$ gets large, and the term ``$II$" representing the vertical curvature of the graph of $u_0$ dominates in the region of $\Omega_C$ close to $C_{11}$.
\end{rem}

We have shown that $u_0$ is a sub-solution of (\ref{GraphEquation}) in $\Omega_{C_0}$. To get a sub-solution of (\ref{GraphEquation}) on all of $\{|y| > |x|\}$ we take a truncation of $u_0$. Namely, the function
$$\underline{u_0} := \text{max}\{0,\, u_0 - \underline{H}(\lambda_{C_0}^{-1-\mu})\}$$
is a sub-solution to (\ref{GraphEquation}) on $\{|y| > |x|\}$, as is 
$$\underline{u_0}_R(\cdot) := R^{-1}\underline{u_0}(R\cdot)$$ 
for any $R > 0$. To conclude this section we let
$$\underline{u} = \underline{u_0}_{R} \text{ in } \{|y| > |x|\}, \quad R = 2^{-\frac{1+\mu}{\mu}},$$
and we extend $\underline{u}$ to all of $\mathbb{R}^4$ by odd reflection over $C_{11}$.
For this choice of $R$, we have in $\{|y| > |x|\}$ that
\begin{equation}\label{Ordering3}
\underline{u} \leq \overline{u}.
\end{equation}
Indeed, in $\{|y| > |x|\}$ we have
\begin{align*}
\underline{u} &= R^{-1}\underline{u_0}(R \cdot) \\
&\leq R^{-1}u_0(R \cdot) \\
&= R^{-1}\underline{H}(\underline{w}) (R \cdot) \\
&\leq R^{-1}\underline{w}(R \cdot) \quad \text{(since } \underline{H}' < 1) \\
&= R^{\mu}\underline{w} \quad \text{(homogeneity)} \\
&\leq \overline{w} \quad \text{(choice of } R \text{ and inequality } (\ref{Ordering2})) \\
&\leq \overline{H}(\overline{w}) \quad \text{(since } \overline{H}' > 1) \\
&= \overline{u}.
\end{align*}
Furthermore, since $\underline{H}$ has linear growth, we have that
$$\sup_{B_r} \underline{u} \geq cr^{1+\mu}$$
for all $r$ large.

\section{Proof of Theorem \ref{Main}}\label{Proof}
Finally, we put everything together to prove the main theorem.

\begin{proof}[{\bf Proof of Theorem \ref{Main}:}]
The definition and uniform ellipticity of $\Phi$ were established in Section \ref{IntegrandChoice}. For $k \geq 1$, solve the Dirichlet problem
$$\Psi_{ij}(\nabla u_k)\partial_i\partial_ju_k = 0, \quad u_k|_{\partial B_k} = \overline{u}.$$
We have the existence of a unique solution $u_k \in C^{\infty}(B_k) \cap C\left(\overline{B_k}\right)$ by the results in Section $5$ of \cite{Si0}. By the symmetries of the integrand $\Psi$ and the boundary data and uniqueness, the functions $u_k$ depend only on $|x|$ and $|y|$, and are odd over $C_{11}$ (note that $\overline{u}$ is odd over $C_{11}$). In particular, they vanish on $C_{11}$. Using (\ref{Ordering3}) and the maximum principle we conclude that
$$\underline{u} \leq u_k \leq \overline{u}$$
in $B_k \cap \{|y| > |x|\}$, with the reverse inequality on the other side of $C_{11}$. Simon's interior gradient estimate (see Section $5$ in \cite{Si0}) implies that for any $R > 0$ and $k > 2R$, the norm $\|u_k\|_{C^1(B_R)}$ is bounded by a constant independent of $k$. We can in fact replace the space $C^1(B_R)$ in this estimate by $C^m(B_R)$ for any $m$, using De Giorgi-Nash-Moser theory and Schauder estimates (\cite{GT}). We may thus extract a subsequence of $\{u_k\}$ that converges locally uniformly along with all its derivatives to a smooth limit $u$ which solves the equation
$$\Psi_{ij}(\nabla u)u_{ij} = 0$$
on $\mathbb{R}^4$ and furthermore satisfies
$$\underline{u} \leq u \leq \overline{u}$$
in $\{|y| > |x|\}$, with the reverse inequality otherwise. Since 
$$\sup_{B_r} \underline{u} \geq cr^{1+\mu}$$ 
for all $r$ large, this completes the proof.
\end{proof}

\section{Discussion}\label{DiscussionSec}

In this final section we discuss how our methods recover, in a systematic way, known examples of nonlinear entire minimal graphs, and we discuss some open questions related to this work.

\subsection{Entire minimal graphs asymptotic to $C_{kk} \times \mathbb{R}$}
The above approach adapts easily to constructing nonlinear entire solutions to equations of minimal surface type whose graphs are asymptotic to $C_{kk} \times \mathbb{R}$ for any $k \geq 1$. In this subsection we outline how to do this in the case of the minimal surface equation and $k \geq 3$, both recovering the examples in \cite{BDG} and giving a systematic way of building super- and sub-solutions starting from a foliation.

For $x,\,y \in \mathbb{R}^{k+1},\, k \geq 3,$ the minimal leaves foliating a side of the cone $C_{kk}$ are dilations of $\{|y| = \sigma(|x|)\}$, where $\sigma$ solves
$$G(\sigma) := \sigma'' + k(1+\sigma'^2)\left(\frac{\sigma'}{\tau} - \frac{1}{\sigma}\right) = 0, \quad \sigma(0) = 1,\, \quad \sigma'(0) = 0.$$
The quantity $G$ is equivalent to the mean curvature of the leaf, up to multiplying by positive constants.
An analysis of this ODE similar to that done above in the more general anisotropic setting shows that for
$$\mu = (k-1/2) - \sqrt{(k-1/2)^2 - 2k},$$
we have
$$\sigma = \tau + a\tau^{-\mu} +O_2(\tau^{-\alpha}) \text{ for some } a > 0.$$
Here
\begin{align*}
\alpha &= \text{min}\{k-1/2 + \sqrt{(k-1/2)^2 - 2k},\, 2\mu + 1\} \\
&= \begin{cases} k-1/2 + \sqrt{(k-1/2)^2 - 2k},\, k = 3 \\
2\mu + 1,\,k \geq 4.
\end{cases}
\end{align*}
From here the analysis follows the same lines. Let $\beta \in (\mu,\,\alpha)$. By analyzing the linearized equation at $\sigma$ with right hand side $\sigma^{-\beta-2}$,
one produces perturbed leaves defined by functions $(\underline{\sigma},\,\overline{\sigma})$ with the asymptotic behavior $\tau + (\underline{a},\,\overline{a})\tau^{-\mu} + O_2(\tau^{-\beta})$ and $\underline{a},\,\overline{a} > 0$, and mean curvature $(-G(\underline{\sigma}),\,G(\overline{\sigma})) \geq c \sigma^{-\beta-2}$. Define $\underline{w},\, \overline{w}$ to be functions that are homogeneous of degree $1 + \mu$ with the perturbed leaves defined by $\underline{\sigma},\,\overline{\sigma}$ as level sets, and then choose $\underline{H},\, \overline{H}$ with linear growth and a constant $K$ such that the minimal surface operator applied to $\overline{H}(\overline{w})$ is nonpositive and to $\max\{0,\,\underline{H}(\underline{w})-K\}$ is nonnegative in $\{|y| > |x|\}$ (recall in this case that $F(s) = \sqrt{1+s^2}$). The analysis is nearly identical after changing the parameters $\gamma,\,M,\,\delta$ from Section \ref{SuperSub} to
$$\gamma = \frac{\mu+1/2}{\mu+1}, \quad M = \frac{2}{\beta - \mu}, \quad \delta = \frac{1}{M(\mu+1)}.$$
 As above, the key terms are a term representing the mean curvature of the level sets (we called it ``$I$" above), which was designed to have a desired sign by perturbing the leaves in the minimal foliation, and terms involving $\underline{H}'',\,\overline{H}''$ (we called them ``$II$" above) which represent a favorable curvature in the remaining ``vertical" direction. These terms dominate in complementary regions of $\{|y| > |x|\}$ (see Remarks \ref{GeomMeaning} and \ref{GeomMeaning2}).

\begin{rem}
It is known more generally by work of Hardt-Simon that each side of any area-minimizing hypercone $C$ with an isolated singularity is foliated by smooth area-minimizing hyperfurfaces \cite{HS}. In this more general setting, perturbing the leaves in the foliation to have mean curvature of a desired sign amounts to solving the Jacobi field equation with a source term decaying at a certain rate. It is feasible that our approach of perturbing the leaves, then stacking could produce entire minimal graphs asymptotic to $C \times \mathbb{R}$ provided the cone has appropriate symmetries; see next sub-section for issues that can arise.
\end{rem}

\subsection{The case $C_{kl},\, k \neq l$}
For any $k,\,l \geq 1$, each side of $C_{kl}$ is foliated by minimizers of a parametric elliptic functional \cite{MY}. Perturbing the leaves in the foliation to have anisotropic mean curvature of a desired sign works in the same way as described above for any $k$ and $l$ (and similarly for the foliations associated to the area-minimizing Lawson cones), as well as the construction of super- and sub-solutions to equations of minimal surface type on each side of the cone that are appropriately ordered and have comparable growth rates. However, to build entire solutions asymptotic to $C_{kk} \times \mathbb{R}$, we rely on the odd symmetry over $C_{kk}$ of solutions to the Dirichlet problem on bounded domains. Thus, another argument is needed to produce examples asymptotic to $C_{kl} \times \mathbb{R}$ when $k \neq l$.

In \cite{Si1} Simon constructs entire minimal graphs that are asymptotic to $C \times \mathbb{R}$ for a large class of area-minimizing cones $C$ including all of the area-minimizing Lawson cones. The approach in that work is first to solve the Dirichlet problem on $B_1$ with large values on one side of the cone and small values on the other, and then note that appropriate rescalings and translations of the resulting graphs converge to complete (but not necessarily entire) minimal graphs. It is then delicately argued that these graphs must in fact be entire, using along the way the minimal foliations constructed in \cite{HS}. Perhaps a combination of approaches in this paper and in \cite{Si1} could elucidate what is happening, and allow the construction of entire anisotropic minimal graphs asymptotic to $C_{kl} \times \mathbb{R}$ in the remaining cases $k \neq l$ and either $k + l < 6$ or $k + l = 6$ and $\text{min}\{k,\,l\} = 1$.

\subsection{Controlled Growth Results}
It is known that minimal graphs satisfying the controlled growth condition $|\nabla u| = o(|x|)$
are necessarily linear \cite{EH}. A key tool in the argument is the Simons identity for the Laplace of the second fundamental form, which has an anisotropic analogue (see \cite{W}). Similar arguments might thus be used to prove controlled growth Bernstein theorems in the anisotropic case, assuming e.g. that $|\nabla u| = O(|x|^{\epsilon})$ for some $\epsilon$ small depending on the dimension and the integrand.

\subsection{Closeness to Area}
Finally, the known results for the area functional are robust under small $C^4$ perturbations of the integrand in (\ref{PEF}) from area on $\mathbb{S}^{n}$. For example, the Bernstein theorem still holds up to dimension $n = 7$ (\cite{Si2}), and the flatness of stable critical points holds in dimension $n = 3$ (this was shown in \cite{CL1},\, \cite{CL2}, also for the first time in the case of the area functional). Interestingly, the latter result holds in dimension $n = 2$ under the hypothesis of $C^2$ closeness to the area functional \cite{Lin}. It would be interesting to determine whether or not the topology in which closeness is measured could be relaxed e.g. to $C^2$ in the other cases. By quantifying the closeness to area of the integrands in this paper or in \cite{MY}, one could gain insights into this question.



\end{document}